\documentclass[12pt,reqno]{amsart}
\usepackage{amsfonts,amssymb,amsbsy,amsmath}

\advance \topmargin by -\headheight
\advance \topmargin by -\headsep
     
\evensidemargin \oddsidemargin
\newtheorem{theorem}{Theorem}[section]
\newtheorem{lemma}[theorem]{Lemma}
\newtheorem{corollary}[theorem]{Corollary}
\newtheorem{conjecture}[theorem]{Conjecture}

\theoremstyle{definition}

\theoremstyle{remark}

\numberwithin{equation}{section}

\newcommand{\mmod}[1]{\,\,(\text{mod}\,\,#1)}

\def\bfa{{\mathbf a}}

\def\bfq{{\mathbf q}}

\def\bfx{{\mathbf x}}

\def\calE{{\mathcal E}} \def\calEhat{{\widehat \calE}}

\def\calM{{\mathcal M}}
\def\calN{{\mathcal N}}

\def\calQ{{\mathcal Q}}

\def\calZ{{\mathcal Z}}

\def\dbC{{\mathbb C}}\def\dbN{{\mathbb N}}
\def\dbR{{\mathbb R}}
\def\dbZ{{\mathbb Z}}

\def\grM{{\mathfrak M}}
\def\grS{{\mathfrak S}}\def\grP{{\mathfrak P}}

\def\grp{{\mathfrak p}}

\def\alp{{\alpha}} 
\def\bet{{\beta}}  
\def\gam{{\gamma}}  
\def\del{{\delta}}

\def\rhotil{{\widetilde \rho}}

\def\Ups{{\Upsilon}}

\def\d{{\partial}}
\def\eps{\varepsilon}

\def\le{\leqslant} \def\ge{\geqslant}

\def\d{{\,{\rm d}}}

\begin{document}
\title[squares and microsquares]{On Linnik's Conjecture:\\ sums of squares and microsquares}
\author[T. D. Wooley]{Trevor D. Wooley}
\address{School of Mathematics, University of Bristol, University Walk, Clifton, Bristol BS8 1TW, United Kingdom}
\email{matdw@bristol.ac.uk}
\date{}
\begin{abstract} We show that almost all natural numbers $n$ not divisible by $4$, and not congruent to $7$ modulo $8$, are represented as the sum of three squares, one of which is the square of an integer no larger than $(\log n)^{1+\eps}$. This answers a conjecture of Linnik for almost all natural numbers, and sharpens a conclusion of Bourgain, Rudnick and Sarnak concerning nearest neighbour distances between normalised integral points on the sphere.\end{abstract}
\maketitle

\section{Introduction} A celebrated theorem of Gauss \cite{Gau1801} shows that every natural number $n$ not of the shape $4^l(8k+7)$ is represented as the sum of three integral squares. For $n\not\equiv 0,4,7\mmod{8}$ these representations are now known to be equidistributed in a suitable sense, as a consequence of work of Duke, Schulze-Pillot, Iwaniec, Golubeva and Fomenko (see \cite{Duk1988}, \cite{DSP1990}, \cite{GF1987}, \cite{Iwa1987}). First conjectured by Linnik, and proved by him \cite{Lin1968} assuming the truth of the Generalised Riemann Hypothesis, this equidistribution property suggests that parallel conclusions may hold in which the variables are restricted in one way or another. Indeed, Linnik conjectured \cite{Lin1968} that subject to local conditions, representations should exist in which one of the squares is small.

\begin{conjecture}[Linnik]\label{conjecture1.1}
Let $\eps$ be a positive number, and suppose that $n$ is odd, $n\not\equiv 7\mmod{8}$ and $n$ is squarefree. Then whenever $n$ is sufficiently large in terms of $\eps$, the Diophantine equation
\begin{equation}\label{1.1}
n=x_1^2+x_2^2+x_3^2
\end{equation}
possesses a solution in which $|x_3|\le n^\eps$.
\end{conjecture}

\noindent Our goal in this paper is to prove that Linnik's Conjecture holds for almost all eligible $n$ in a particularly strong form.\par

When $Y$ is positive, denote by $R(n;Y)$ the number of representations of the integer $n$ in the shape (\ref{1.1}) with $\bfx\in \dbN^3$ and $x_3\le Y$. Also, write $E(X;Y)$ for the number of integers $n$, with $X/2<n\le X$, $4\nmid n$ and $n\not\equiv 7\mmod{8}$, having no such representation, so that $R(n;Y)=0$.

\begin{theorem}\label{theorem1.2}
Suppose that $0<\del<\tfrac{1}{6}$. Then whenever
$$(\log X)(\log \log X)^2\le Y\le (\log X)^{1+\del},$$
one has $E(X;Y)\ll XY^{-1}(\log X)(\log \log X)^2$.
\end{theorem}

It follows that an enhanced version of Linnik's Conjecture holds for almost all eligible integers $n$.

\begin{corollary}\label{corollary1.3}
Let $\eps$ be a positive number. Then for almost all natural numbers $n$ with $4\nmid n$ and $n\not\equiv 7\mmod{8}$, the Diophantine equation (\ref{1.1}) possesses a solution $\bfx\in \dbN^3$ in which $x_3\le (\log n)(\log \log n)^{2+\eps}$.
\end{corollary}

By analogy with recent work of the author joint with Br\"udern \cite{BW2010} concerning sums of four cubes, it would be natural to refer to the small square as a {\it minisquare}. However, we are able to take this square so small that the term {\it microsquare} seems more appropriate, explaining the title of the paper.\par

For a large eligible integer $n$, denote by $\calEhat(n)$ the set of points $n^{-1/2}\bfx$, as $\bfx$ runs over the integral solutions of the equation (\ref{1.1}). Bourgain, Rudnick and Sarnak \cite{BRS2012} consider the minimum spacing of the points $P_i\in \calEhat(n)$ by means of the function
$$m(\calEhat(n))=\min \{ |P_i-P_j|:\text{$P_i,P_j\in \calEhat(n)$, $P_i\ne P_j$}\},$$
in which $|P_i-P_j|$ denotes the Euclidean distance from $P_i$ to $P_j$. In particular, they show that for each $\eps>0$, one has $m(\calEhat(n))\ll n^{\eps-1}$ for almost all eligible $n$ (see \cite[Corollary 1.7]{BRS2012}). As a consequence of Corollary \ref{corollary1.3}, one may sharpen this conclusion.

\begin{corollary}\label{corollary1.4}
For each $\eps>0$, one has $m(\calEhat(n))\ll n^{-1}(\log n)^{1+\eps}$ for almost all natural numbers $n$ with $4\nmid n$ and $n\not\equiv 7\mmod{8}$.
\end{corollary}

Since the average gap between neighbouring sums of two squares of size $X$ grows in proportion to $\sqrt{\log X}$, it follows that when $Y=o(\sqrt[4]{\log X})$ one must have $E(X;Y)\gg X$. Consequently, no analogue of Theorem \ref{theorem1.2} is possible when the parameter $Y$ is replaced by a function growing more slowly than $\sqrt[4]{\log X}$. Work of Richards \cite{Ric1982}, meanwhile, shows that for arbitrarily large positive numbers $X$, there exist gaps of size nearly $\tfrac{1}{4}\log X$ between successive sums of two squares of size $X$. Thus, whenever $Y$ grows more slowly than $\sqrt{\log X}$, one has $E(X;Y)\rightarrow \infty$ as $X\rightarrow \infty$.\par

We prove Theorem \ref{theorem1.2} by means of the Hardy-Littlewood method. The minor arc analysis, although strictly speaking conventional in nature, is motivated by recent work on ``slim'' exceptional sets (see, for example, the paper \cite{Woo2002b}). The major arc analysis is complicated by difficulties associated with what superficially appears to be a divergent singular series. Owing to the precise control available for quadratic Gauss sums, this hurdle is surmounted with only modest complications. A similar though more straightforward argument yields a conclusion related to that of Corollary \ref{corollary1.3} for sums of four squares.

\begin{theorem}\label{theorem1.5}
Let $\eps$ be a positive number. Then for almost all natural numbers $n$ with $8\nmid n$, the Diophantine equation
\begin{equation}\label{1.2}
n=x_1^2+x_2^2+x_3^2+x_4^2
\end{equation}
possesses a solution $\bfx\in \dbN^4$ in which $\max\{x_3,x_4\}\le (\log n)^{1/2}(\log \log n)^{3/2+\eps}$.
\end{theorem}

We note that numerous authors have sought conclusions related to that of Gauss on sums of three squares in which the variables are restricted in various ways. For conclusions involving sums of three squares of almost primes, smooth numbers and squarefree numbers, we direct the reader to \cite{BB2005}, \cite{BBD2009} and \cite{Bak2006}, respectively. If one is prepared to tolerate the potential existence of an exceptional set of integers, thin amongst those integers constrained by congruence conditions to be eligible for representation, then sums of three squares of primes are also accessible (see \cite{Hua1938}, and \cite{HK2010} for the latest conclusions).\par

We finish by recording a convention concerning the use of the number $\eps$. Whenever $\eps$ appears in a statement, either implicitly or explicitly, we assert that the statement holds for each $\eps>0$. Note that the ``value'' of $\eps$ may consequently change from statement to statement. In addition, we write $p^h\| n$ when $p^h|n$ but $p^{h+1}\nmid n$.\par

The author is grateful to Zeev Rudnick for an enquiry which led to the main result of this paper.

\section{Preparatory manoeuvres} The method that we employ to prove Theorem \ref{theorem1.2} is based on the Hardy-Littlewood (circle) method. Our object in this section is to initiate the application of this method, reaching the point at which it is
 apparent what auxiliary estimates will be necessary to complete the analysis. Let $\del>0$. We consider a positive number $X$ sufficiently large in terms of $\del$, and we take $Y$ to be a real number with
$$(\log X)(\log \log X)^2\le Y\le (\log X)^{1+\del}.$$
We denote by $\calZ(X)$ the set of integers $n$ with
\begin{equation}\label{2.1}
X/2<n\le X,\quad 4\nmid n\quad \text{and}\quad n\not\equiv 7\mmod{8},
\end{equation}
for which $R(n;Y)=0$, and we abbreviate $\text{card}(\calZ(X))$ to $Z$. Write $P$ for $[X^{1/2}]$, and define the exponential sums $f(\alp)$ and $g(\alp)$ by
$$f(\alp)=\sum_{P/2<x\le P}e(\alp x^2)\quad \text{and}\quad g(\alp)=\sum_{1\le y\le Y}e(\alp y^2).$$
Here, as usual, we write $e(z)$ for $e^{2\pi iz}$.\par

We put
\begin{equation}\label{2.2}
W=(\log X)^{1/5}.
\end{equation}
We then take $\grP$ to be the union of the intervals
$$\grP(q,a)=\{ \alp\in [0,1):|\alp-a/q|\le WX^{-1}\},$$
with $0\le a\le q\le W$ and $(a,q)=1$. Finally, we wite $\grp=[0,1)\setminus \grP$.\par

Next, write
$$S(q,a)=\sum_{r=1}^qe(ar^2/q)\quad \text{and}\quad v(\bet)=\int_{P/2}^Pe(\bet \gam^2)\d\gam ,$$
and define $f^*(\alp)$ and $g^*(\alp)$ for $\alp\in \grP(q,a)\subseteq \grP$ by putting
\begin{equation}\label{2.3}
f^*(\alp)=q^{-1}S(q,a)v(\alp-a/q)\quad \text{and}\quad g^*(\alp)=q^{-1}S(q,a)Y.
\end{equation}
Then it follows from \cite[Theorem 4.1]{Vau1997} that whenever $\alp\in \grP(q,a)\subseteq \grP$, one has
$$f(\alp)-f^*(\alp)\ll q^{1/2+\eps}\ll W^{1/2+\eps},$$
and by making use of Taylor's theorem in combination with \cite[Theorem 4.1]{Vau1997}, one finds that
$$g(\alp)-g^*(\alp)\ll q^{1/2+\eps}+1+Y^3WX^{-1}\ll W^{1/2+\eps}.$$
The measure of $\grP$ is $O(W^3X^{-1})$, and thus we deduce that for all integers $n$ one has
$$\int_\grP f(\alp)^2g(\alp)e(-n\alp)\d\alp -\int_\grP f^*(\alp)^2g^*(\alp)e(-n\alp)\d\alp \ll P^2W^4X^{-1}.$$
A routine calculation leads from (\ref{2.3}) to the formula
$$\int_\grP f^*(\alp)^2g^*(\alp)e(-n\alp)\d\alp =\grS(n;W)J(n;W),$$
where
\begin{equation}\label{2.4}
\grS(n;W)=\sum_{1\le q\le W}A(q;n),
\end{equation}
in which we have written
\begin{equation}\label{2.5}
A(q;n)=\sum^q_{\substack{a=1\\ (a,q)=1}}q^{-3}S(q,a)^3e(-na/q),
\end{equation}
and
$$J(n;W)=Y\int_{-WX^{-1}}^{WX^{-1}}v(\bet)^2e(-\bet n)\d\bet .$$
We may therefore conclude thus far that
\begin{equation}\label{2.6}
\int_\grP f(\alp)^2g(\alp)e(-n\alp)\d\alp =\grS(n;W)J(n;W)+O(YW^{-1}).
\end{equation}

\par We next recall from \cite[Lemma 6.2]{Vau1997} that
$$v(\bet)\ll P(1+|\bet |X)^{-1},$$
so that the singular integral $J(n;W)$ converges absolutely as $W\rightarrow \infty$. When $X/2<n\le X$, the discussion concluding \cite[Chapter 4]{Dav1963} consequently delivers the asymptotic relation
\begin{equation}\label{2.7}
Y\ll J(n;W)\ll Y.
\end{equation}
We combine this asymptotic information concerning the truncated singular integral with a corresponding lower bound for the truncated singular series. This we prepare in \S\S3 and 4, delivering the bound contained in the following lemma.

\begin{lemma}\label{lemma2.1} Let $\del$ be a sufficiently small positive number. Then for each integer $n$ with $X/2<n\le X$, $4\nmid n$ and $n\not\equiv 7\mmod{8}$, one has the lower bound $\grS(n;W)\ge \del (\log W)^{-1}$, with at most $XW^{\eps-1}$ possible exceptions.
\end{lemma}

On combining (\ref{2.6}), (\ref{2.7}) and the conclusion of Lemma \ref{lemma2.1}, we deduce that for each integer $n$ subject to the conditions (\ref{2.1}), one has
\begin{equation}\label{2.8}
\int_\grP f(\alp)^2g(\alp)e(-n\alp)\d\alp \gg Y(\log W)^{-1},
\end{equation}
with at most $XW^{\eps-1}$ possible exceptions.\par

Let $\calN$ denote the set of integers $n$ subject to the conditions (\ref{2.1}) for which the lower bound (\ref{2.8}) holds. Also, define $\calZ^*(X)$ to be the set of integers $n$ with $n\in \calN$ for which $R(n;Y)=0$, and 
abbreviate $\text{card}(\calZ^*(X))$ to $Z^*$. Then we have
\begin{equation}\label{2.9}
Z\le Z^*+XW^{\eps-1}.
\end{equation}
For each integer $n\in \calZ^*(X)$, we have $R(n;Y)=0$, and so it follows from orthogonality that
$$\int_0^1f(\alp)^2g(\alp)e(-n\alp)\d\alp =0,$$
whence
\begin{equation}\label{2.10}
\int_\grP f(\alp)^2g(\alp)e(-n\alp)\d\alp =-\int_\grp f(\alp)^2g(\alp)e(-n\alp)\d\alp .
\end{equation}
Define the exponential sum $K(\alp)$ by putting
$$K(\alp)=\sum_{n\in \calZ^*(X)}e(-n\alp).$$
Then it follows from (\ref{2.8}) and (\ref{2.10}) that
\begin{align*}
Z^*Y(\log W)^{-1}&\ll \sum_{n\in \calZ^*(X)}\int_\grP f(\alp)^2g(\alp)e(-n\alp)\d\alp \\
&\le \Bigl| \sum_{n\in \calZ^*(X)}\int_\grp f(\alp)^2g(\alp)e(-n\alp)\d\alp \Bigr|,
\end{align*}
whence
\begin{equation}\label{2.11}
Z^*Y(\log W)^{-1}\ll \int_\grp |f(\alp)^2g(\alp)K(\alp)|\d\alp .
\end{equation}
This relation is the starting point for our estimation of $Z$, which we achieve in \S5 by bounding the integral on the right hand side of (\ref{2.11}).

\section{The singular series: local factors} Our proof of Lemma \ref{lemma2.1} proceeds in two steps, the first being the computation of local factors implicit in the definition of $\grS(n;W)$. In principle this step is entirely classical in nature, though we have been unable to locate a convenient reference in the literature. We begin by recalling some properties of exponential sums. In this context, we use the convention that the letter $p$ always denotes a prime number.

\begin{lemma}\label{lemma3.1} Let $p$ be an odd prime, and suppose that $(a,p)=1$. Then for each natural number $l$ one has
$$S(p^{2l},a)=p^l\quad \text{and}\quad S(p^{2l+1},a)=p^lS(p,a).$$
\end{lemma}

\begin{proof} This is immediate from \cite[Lemma 4.4]{Vau1997}.
\end{proof}

Write $\chi_p(b)$ for the familiar Legendre symbol $\Bigl( {\displaystyle{\frac{b}{p}}}\Bigr)$. We require the following well-known result on the average of the Legendre symbol over quadratic polynomials.

\begin{lemma}\label{lemma3.2} Let $f(x)=ax^2+bx+c$, where $a,b,c$ are integers, and let $p$ be an odd prime with $(a,p)=1$. Then
$$\sum_{x=1}^p\chi_p(f(x))=\begin{cases} -\chi_p(a),&\text{when $p\nmid (b^2-4ac)$,}\\
(p-1)\chi_p(a),&\text{when $p|(b^2-4ac)$.}\end{cases}$$
\end{lemma}

Finally, we make use of the Ramanujan sum
$$c_q(m)=\sum^q_{\substack{a=1\\ (a,q)=1}}e(am/q).$$

\begin{lemma}\label{lemma3.3} One has
$$c_q(m)=\frac{\mu(q/(q,m))\phi(q)}{\phi(q/(q,m))}.$$
\end{lemma}

\begin{proof} This is \cite[Theorem 272]{HW1979}.\end{proof}

We are now equipped to evaluate $A(p^h;n)$ when $p$ is odd.

\begin{lemma}\label{lemma3.4} Let $p$ be an odd prime. Then for each natural number $l$ one has
$$A(p^{2l};n)=\begin{cases} p^{-1-l}(p-1),&\text{when $p^{2l}|n$,}\\
-p^{-1-l},&\text{when $p^{2l-1}\|n$,}\\
0,&\text{when $p^{2l-1}\nmid n$.}\end{cases}$$
\end{lemma}

\begin{proof} On recalling (\ref{2.5}), it follows from Lemma \ref{lemma3.1} that
$$A(p^{2l};n)=\sum^{p^{2l}}_{\substack{a=1\\ (a,p)=1}}(p^{-l})^3e(-na/p^{2l})=p^{-3l}c_{p^{2l}}(-n).$$
Notice that
$$\mu(p^{2l}/(p^{2l},n))=\begin{cases}1,&\text{when $p^{2l}|n$,}\\
-1,&\text{when $p^{2l-1}\|n$,}\\
0,&\text{when $p^{2l-1}\nmid n$.}\end{cases}$$
Then the conclusion of the lemma follows directly from Lemma \ref{lemma3.3} in the respective cases.
\end{proof}

\begin{lemma}\label{lemma3.5}
Let $p$ be an odd prime. Then for each non-negative integer $l$ one has
$$A(p^{2l+1};n)=\begin{cases} p^{-l-1}\chi_p(-n/p^{2l}),&\text{when $p^{2l}\| n$,}\\
0,&\text{otherwise.}\end{cases}$$
\end{lemma}

\begin{proof} Again recalling (\ref{2.5}), we deduce from Lemma \ref{lemma3.1} that
\begin{align*}
A(p^{2l+1};n)&=\sum^{p^{2l+1}}_{\substack{a=1\\ (a,p)=1}}(p^{-l-1}S(p,a))^3e(-na/p^{2l+1})\\
&=p^{-3l-3}\sum^p_{\substack{b=1\\(b,p)=1}}\sum_{c=1}^{p^{2l}}S(p,b)^3e(-n(b+pc)/p^{2l+1})\\
&=p^{-3l-3}\sum^p_{\substack{b=1\\ (b,p)=1}}S(p,b)^3e(-nb/p^{2l+1})\sum_{c=1}^{p^{2l}}e(-nc/p^{2l}).
\end{align*}
The last sum is zero unless $p^{2l}|n$, in which case one deduces that
\begin{equation}\label{3.1}
A(p^{2l+1};n)=p^{-l-3}\sum^p_{\substack{b=1\\ (b,p)=1}}S(p,b)^3e(-(n/p^{2l})b/p).
\end{equation}

\par Write $T(m)$ for the number of solutions of the congruence
$$x_1^2+x_2^2+x_3^2\equiv m\mmod{p},$$
with $1\le x_i\le p$ $(1\le i\le 3)$. Then it follows from (\ref{3.1}) via orthogonality that
\begin{equation}\label{3.2}
A(p^{2l+1};n)=p^{-l-3}(pT(n/p^{2l})-p^3).
\end{equation}
On the other hand, one has
$$T(m)=\sum_{x_1=1}^p\sum_{x_2=1}^p(1+\chi_p(m-x_1^2-x_2^2)),$$
so that it follows from Lemma \ref{lemma3.2} that
\begin{align*}
T(m)&=p^2+\sum^p_{\substack{x_1=1\\ x_1^2\equiv m\mmod{p}}}(p-1)\chi_p(-1)-\sum^p_{\substack{x_1=1\\ x_1^2\not\equiv m\mmod{p}}}\chi_p(-1)\\
&=p^2+p\chi_p(m)\chi_p(-1)=p^2+p\chi_p(-m).
\end{align*}
The conclusion of the lemma now follows by substituting this formula for $T(m)$ into (\ref{3.2}).
\end{proof}

By combining the conclusions of Lemmata \ref{lemma3.4} and \ref{lemma3.5}, we obtain a lower bound for the $p$-adic density factor in the singular series.

\begin{lemma}\label{lemma3.6}
Let $p$ be an odd prime. Then for each natural number $H$, one has
$$\sum_{h=0}^HA(p^h;n)\ge 1-p^{-1}.$$
\end{lemma}

\begin{proof} Write $L=[H/2]$ and $M=[(H-1)/2]$. Then
$$\sum_{h=0}^HA(p^h;n)=1+\sum_{l=1}^LA(p^{2l};n)+\sum_{m=0}^MA(p^{2m+1};n).$$
Suppose first that $p^{2\nu}\|n$. Then we deduce from Lemma \ref{lemma3.4} that when $L\ge 1$ and $\nu\ge 1$ one has
$$\sum_{l=1}^LA(p^{2l};n)=\sum_{l=1}^{\max\{L,\nu\}}p^{-1-l}(p-1)=p^{-1}-p^{-1-\max\{L,\nu\}}\ge p^{-1}-p^{-2},$$
whilst for $L\ge 1$ and $\nu=0$ instead
$$\sum_{l=1}^LA(p^{2l};n)=0.$$
Meanwhile, from Lemma \ref{lemma3.5} one finds that
$$\sum_{m=0}^MA(p^{2m+1};n)\ge -p^{-\nu-1}.$$
We may therefore conclude when $p^{2\nu}\|n$ that, subject to the condition
$$\min\{H,2\nu\}\ge 2,$$
one has
\begin{equation}\label{3.3}
\sum_{h=0}^HA(p^h;n)\ge 1+(p^{-1}-p^{-2})-p^{-\nu-1}\ge 1,
\end{equation}
whilst in all other situations one has instead
\begin{equation}\label{3.4}
\sum_{h=0}^H A(p^h;n)\ge 1-p^{-\nu-1}\ge 1-p^{-1}.
\end{equation}

\par Suppose next that $p^{2\nu-1}\|n$ with $\nu\ge 1$. In this case it follows from Lemma \ref{lemma3.5} that
$$\sum_{m=0}^MA(p^{2m+1};n)=0.$$
From Lemma \ref{lemma3.4}, meanwhile, when $L\ge 1$ one obtains the relation
$$\sum_{l=1}^LA(p^{2l};n)\ge \sum_{l=1}^{\max\{L,\nu-1\}}p^{-1-l}(p-1)-p^{-1-\nu}\ge -p^{-1-\nu}.$$
We therefore conclude that when $p^{2\nu-1}\|n$ and $\nu\ge 1$, then
\begin{equation}\label{3.5}
\sum_{h=0}^HA(p^h;n)\ge 1-p^{-2}.
\end{equation}

\par On collecting together the estimates (\ref{3.3}), (\ref{3.4}) and (\ref{3.5}), we deduce that
$$\sum_{h=0}^HA(p^h;n)\ge 1-p^{-1},$$
thereby confirming the conclusion of the lemma.
\end{proof}

We have yet to estimate the $2$-adic factor.

\begin{lemma}\label{lemma3.7}
Suppose that $n$ is a natural number with $4\nmid n$ and $n\not\equiv 7 \mmod{8}$. Then for each natural number $H$ with $H\ge 3$, one has
$$\sum_{h=0}^HA(2^h;n)\ge 2^{-6}.$$
\end{lemma}

\begin{proof} It is apparent that when $n\equiv 1,2,3,5,6\mmod{8}$, then there is a solution of the congruence
$$x_1^2+x_2^2+x_3^2\equiv n\mmod{8},$$
in which $x_1$ is odd. It is then a consequence of \cite[Lemmata 2.12 and 2.13]{Vau1997} that
$$\sum_{h=0}^H A(2^h;n)\ge (2^H)^{-2}(2^{H-3})^2=2^{-6}.$$
This completes the proof of the lemma.
\end{proof}

In order to define our truncated product of local densities, for each prime number $p$ we define the exponent $H(p)$ to be the largest integer $H$ for which $p^H\le W$. We then put
\begin{equation}\label{3.6}
\grS^*(n;W)=\prod_{p\le W}\sum_{h=0}^{H(p)}A(p^h;n).
\end{equation}

\begin{lemma}\label{lemma3.8} For each natural number $n$ with $4\nmid n$ and $n\not\equiv 7\mmod{8}$, one has $\grS^*(n;W)\gg (\log W)^{-1}$.
\end{lemma}

\begin{proof} The conclusions of Lemmata \ref{lemma3.6} and \ref{lemma3.7} demonstrate that
$$\grS^*(n;W)\ge 2^{-6}\prod_{2<p\le W}(1-p^{-1}).$$
The product on the right hand side here may be bounded via Mertens' formula, and thus one obtains
$$\grS^*(n;W)\ge 2^{-6}(2e^{-\gam}+o(1))(\log W)^{-1}\gg (\log W)^{-1}.$$
This completes the proof of the lemma.
\end{proof}

\section{The singular series: comparing truncations}
The singular series $\grS^*(n;W)$ is defined in (\ref{3.6}) via a multiplicative truncation, whereas the corresponding series $\grS(n;W)$ defined in (\ref{2.4}) might be thought of as being given by an additive truncation. Our goal in this section is to show that, with a small number of exceptional integers $n$, the two series are close to one another, thereby establishing Lemma \ref{lemma2.1}. A similar though more elaborate argument is described in \cite{Vau1980}. There, for a related problem, one experiences additional complications. We examine the difference between $\grS^*(n;W)$ and $\grS(n;W)$ in mean square by means of the following lemma.

\begin{lemma}\label{lemma4.1} Let $X$, $W$ and $Q$ be large positive numbers with $W<Q$. Then whenever $\calQ\subseteq (W,Q]\cap \dbZ$, one has the estimate
$$\sum_{X/2<n\le X}\Bigl| \sum_{q\in \calQ}A(q;n)\Bigr|^2\ll XW^{-1}+Q(\log Q)^2.$$
\end{lemma}

\begin{proof} On recalling (\ref{2.5}), we find that the sum in question is equal to
$$\Ups_0=\sum_{q_1,q_2\in \calQ}\sum^{q_1}_{\substack{a_1=1\\ (a_1,q_1)=1}}\sum^{q_2}_{\substack{a_2=1\\ (a_2,q_2)=1}}(q_1q_2)^{-3}S(q_1,a_1)^3S(q_2,-a_2)^3T(\bfq,\bfa),$$
in which we have written
$$T(\bfq,\bfa)=\sum_{X/2<n\le X}e\Bigl( n\Bigl( \frac{a_2}{q_2}-\frac{a_1}{q_1}\Bigr)\Bigr).$$
It follows from the proof of \cite[Theorem 4.2]{Vau1997} that when $(a,q)=1$ one has $|S(q,a)|^2\le 2q$. Moreover, one has
$$T(\bfq,\bfa)\ll \min\Bigl\{ X,\Bigl\| \frac{a_2}{q_2}-\frac{a_1}{q_1}\Bigr\|^{-1}\Bigr\} .$$
We therefore see that 
\begin{equation}\label{4.1}
\Ups_0\ll \sum_{W<q_1,q_2\le Q}(q_1q_2)^{-3/2}\sum^{q_1}_{\substack{a_1=1\\ (a_1,q_1)=1}}\sum^{q_2}_{\substack{a_2=1\\ (a_2,q_2)=1}}\min \Bigl\{ X,\Bigl\| \frac{a_2}{q_2}-\frac{a_1}{q_1}\Bigr\|^{-1}\Bigr\} .
\end{equation}

\par The diagonal terms on the right hand side of (\ref{4.1}), in which $a_1/q_1=a_2/q_2$, whence $a_1=a_2$ and $q_1=q_2$, make a contribution $\Ups_1$, where
\begin{equation}\label{4.2}
\Ups_1\ll X\sum_{W<q\le Q}\sum_{a=1}^qq^{-3}\ll XW^{-1}.
\end{equation}
Meanwhile, the off-diagonal terms with $a_1/q_1\ne a_2/q_2$ may be classified according to the greatest common divisor $d$ of $q_1$ and $q_2$, and the corresponding least common multiple $q$. Thus one finds that the off-diagonal terms make a contribution $\Ups_2$, where
\begin{align}
\Ups_2&\ll \sum_{1\le d\le Q}\sum_{1\le q\le Q^2/d}(dq)^{-3/2}d\sum^q_{\substack{b=1\\ (b,q)=1}}\min\{ X,\|b/q\|^{-1}\} \notag \\
&\ll \sum_{1\le d\le Q}d^{-1/2}\sum_{1\le q\le Q^2/d}q^{-3/2}(q\log q)\notag \\
&\ll \sum_{1\le d\le Q}d^{-1/2}(Q^2/d)^{1/2}\log Q\ll Q(\log Q)^2.\label{4.3}
\end{align}

Combining the contributions (\ref{4.2}) and (\ref{4.3}) within (\ref{4.1}), we obtain the estimate claimed in the conclusion of the lemma.
\end{proof}

We apply Lemma \ref{lemma4.1} with $Q$ defined by
$$Q=\prod_{p\le W}p^{H(p)},$$
in which $H(p)$ is defined as in the preamble to Lemma \ref{lemma3.8}. It follows from an application of the Prime Number Theorem with error term that
$$\log Q=\sum_{p\le W}H(p)\log p\le \sum_{p\le W}\log W=W+O(W/\log W)\le 2W,$$
so that $Q\le e^{2W}$. We then define the set $\calQ$ by
$$\calQ=\{ q\in (W,Q]:\text{for all primes $p$, $p^h\|q\Rightarrow p^h\le W$}\}.$$
By the multiplicative property of $A(q;n)$ that follows, for example, from \cite[Lemma 2.11]{Vau1997}, it follows from (\ref{2.4}) and (\ref{3.6}) that
\begin{align*}
\grS^*(n;W)&=\prod_{p\le W}\sum_{h=0}^{H(p)}A(p^h;n)=\sum_{1\le q\le W}A(q;n)+\sum_{q\in \calQ}A(q;n)\\
&=\grS(n;W)+\sum_{q\in \calQ}A(q;n).
\end{align*}
Thus, from Lemma \ref{lemma4.1} we deduce that
\begin{equation}\label{4.4}
\sum_{X/2<n\le X}|\grS^*(n;W)-\grS(n;W)|^2\ll XW^{-1}+W^2e^{2W}.
\end{equation}

\par Suppose, if possible, that
$$|\grS^*(n;W)-\grS(n;W)|>(\log W)^{-2}$$
for a set of integers $\calE\subseteq (X/2,X]$. Then it follows that
$$\sum_{X/2<n\le X}|\grS^*(n;W)-\grS(n;W)|^2>(\log W)^{-4}\text{card}(\calE),$$
whence from (\ref{4.4}) we have
$$\text{card}(\calE)\ll (\log W)^4(XW^{-1}+W^2e^{2W}).$$
On recalling from (\ref{2.2}) that $W=(\log X)^{1/5}$, we conclude that
$$\text{card}(\calE)\ll X(\log W)^4W^{-1}.$$
Thus, for all integers $n$ with $X/2<n\le X$, one has the upper bound
$$|\grS^*(n;W)-\grS(n;W)|\le (\log W)^{-2},$$
with at most $XW^{\eps-1}$ possible exceptions. But Lemma \ref{lemma3.8} shows that for a sufficiently small positive number $\del$, one has $\grS^*(n;W)\ge 2\del (\log W)^{-1}$ for every natural number $n$ with $4\nmid n$ and $n\not\equiv 7\mmod{8}$. Consequently, for all integers $n$ with $X/2<n\le X$, $4\nmid n$ and $n\not\equiv 7\mmod{8}$, one has the lower bound
$$\grS(n;W)\ge 2\del (\log W)^{-1}-(\log W)^{-2}\ge \del (\log W)^{-1},$$
with at most $XW^{\eps-1}$ possible exceptions, and Lemma \ref{lemma2.1} follows at once.

\section{The minor arc treatment}
Our argument thus far suffices to establish the upper bound (\ref{2.11}). All that remains is to estimate satisfactorily the minor arc contribution on the right hand side of (\ref{2.11}). It transpires that this is dominated by a diagonal contribution. We first prepare a lemma motivated by Br\"udern's pruning lemma \cite[Lemma 2]{Bru1988}.

\begin{lemma}\label{lemma5.1} Let $N$, $Q$ and $R$ be real numbers satisfying $R\le Q\le N$. For $0\le a\le q\le Q$ and $(a,q)=1$, let $\calM(q,a)$ denote a subset of the interval $[a/q-\frac{1}{2},a/q+\frac{1}{2}]$. Suppose further that for each $\alp \in \calM(q,a)$, one has
$$q+N|q\alp-a|>R.$$
Write $\calM$ for the union of all $\calM(q,a)$, and let $G:\calM\rightarrow \dbC$ be a function which for $\alp \in \calM(q,a)$ satisfies
$$G(\alp)\ll (q+N|q\alp-a|)^{-2}.$$
Furthermore, let $\Psi:\dbR\rightarrow [0,\infty)$ be a function with a Fourier expansion
$$\Psi(\alp)=\sum_{|h|\le H}\psi_he(\alp h)$$
such that $\log H\ll \log N$. Then
$$\int_\calM G(\alp)\Psi(\alp)\d\alp \ll N^{-1}\Bigl( \psi_0\log (2Q/R)+H^\eps R^{\eps-1}\sum_{h\ne 0}|\psi_h|\Bigr).$$
\end{lemma}

\begin{proof} Following the argument of the proof of \cite[Lemma 2]{Bru1988}, mutatis mutandis, we see that
$$\int_\calM G(\alp)\Psi(\alp)\d\alp \ll T_1+T_2,$$
where
$$T_1=\sum_{R/2<q\le Q}\sum^q_{\substack{a=1\\ (a,q)=1}} \int_{-1/2}^{1/2}q^{-2}(1+N|\bet|)^{-2}\Psi(\bet+a/q)\d\bet $$
and
$$T_2=\sum_{1\le q\le R/2}\sum^q_{\substack{a=1\\ (a,q)=1}}\  \int\limits_{|\bet|>R/(2qN)}q^{-2}(1+N|\bet|)^{-2}\Psi(\bet+a/q)\d\bet .$$

\par Write
$$\rho_h=\int_{-1/2}^{1/2}(1+N|\bet|)^{-2}e(\bet h)\d\bet ,$$
so that $\rho_h\ll N^{-1}$. Then one has
$$T_1\ll \sum_{R/2<q\le Q}\sum^q_{\substack{a=1\\ (a,q)=1}}\sum_{|h|\le H}q^{-2}\psi_h\rho_he(ah/q)=\sum_{R/2<q\le Q}\sum_{|h|\le H}q^{-2}\psi_h\rho_hc_q(h).$$
Here, the Ramanujan sum $c_q(h)$ satisfies the bound
$$c_q(h)\le \sum_{d|(q,h)}d,$$
as a consequence of \cite[Theorem 271]{HW1979}. On noting also the trivial bound $c_q(0)\le q$, we thus obtain the estimate
\begin{align*}
T_1&\ll \sum_{R/2<q\le Q}q^{-1}\psi_0\rho_0+\sum_{0<|h|\le H}|\psi_h\rho_h|\sum_{R/2<q\le Q}\sum_{d|(q,h)}dq^{-2}\notag \\
&\ll N^{-1}\Bigl( \psi_0\log (2Q/R)+\sum_{0<|h|\le H}|\psi_h|\sum_{d|h}d^{-1}\sum_{R/(2d)<r\le Q/d}r^{-2}\Bigr) \notag \\
&\ll N^{-1}\Bigl( \psi_0\log (2Q/R)+\sum_{0<|h|\le H}|\psi_h|\sum_{d|h}d^{-1}(2d/R)\Bigr) .
\end{align*}
We therefore conclude that
\begin{equation}\label{5.1}
T_1\ll N^{-1}\Bigl( \psi_0\log (2Q/R)+H^\eps R^{-1}\sum_{0<|h|\le H}|\psi_h|\Bigr) .
\end{equation}

\par Similarly, one finds that
\begin{align*}
T_2&\ll \sum_{1\le q\le R/2}\sum^q_{\substack{a=1\\ (a,q)=1}}\sum_{|h|\le H}q^{-2}\psi_h\rhotil_h(q)e(ah/q)\\
&=\sum_{1\le q\le R/2}\sum_{|h|\le H}q^{-2}\psi_h\rhotil_h(q)c_q(h),
\end{align*}
where
$$\rhotil_h(q)=\int\limits_{|\bet|>R/(2qN)}(1+N|\bet|)^{-2}e(\bet h)\d\bet \ll N^{-1}(q/R).$$
Thus we obtain the estimate
\begin{align*}
T_2&\ll \sum_{1\le q\le R/2}q^{-1}\psi_0\rhotil_0(q)+\sum_{0<|h|\le H}|\psi_h|\sum_{1\le q\le R/2}|\rhotil_h(q)|\sum_{d|(q,h)}dq^{-2}\\
&\ll N^{-1}\Bigl( \psi_0+R^{-1}\sum_{0<|h|\le H}|\psi_h|\sum_{d|h}\sum_{1\le r\le R/(2d)}r^{-1}\Bigr) .\end{align*}
In this instance, therefore, we conclude that
\begin{equation}\label{5.2}
T_2\ll N^{-1}\Bigl( \psi_0+H^\eps R^{\eps-1}\sum_{0<|h|\le H}|\psi_h|\Bigr) .
\end{equation}
The conclusion of the lemma follows by collecting together (\ref{5.1}) and (\ref{5.2}).
\end{proof}

We apply this conclusion to derive an auxiliary minor arc estimate.

\begin{lemma}\label{lemma5.2} One has
$$\int_\grp |f(\alp)^4g(\alp)^2|\d\alp \ll XY\log X+XY^{2+\eps}W^{\eps-1}.$$
\end{lemma}

\begin{proof} As a consequence of Dirichlet's Theorem on Diophantine approximation, the unit interval $[0,1)$ is contained in the union $\grM$ of the arcs
$$\grM(q,a)=\{ \alp \in [0,1):|q\alp-a|\le X^{-1/2}\},$$
with $0\le a\le q\le X^{1/2}$ and $(a,q)=1$. If $\alp\in \grp$, moreover, then whenever $0\le a\le q\le W$ and $(a,q)=1$, one necessarily has $|\alp-a/q|>WX^{-1}$. According to \cite[Theorem 4]{Vau2009}, when $\alp\in \grM(q,a)\subseteq \grM$, one has
\begin{align*}
f(\alp)&\ll P(q+P^2|q\alp-a|)^{-1/2}+(q+P^2|q\alp-a|)^{1/2}\\
&\ll P(q+X|q\alp-a|)^{-1/2}.
\end{align*}
Thus, whenever $\alp\in \grM(q,a)\subseteq \grM$ we have
$$|f(\alp)|^4\ll X^2(q+X|q\alp-a|)^{-2}.$$

\par We apply Lemma \ref{lemma5.1} with the arcs $\calM(q,a)$ taken to be $\grM(q,a)\cap \grp$, and with
$$Q=X^{1/2},\quad R=W,\quad G(\alp)=X^{-2}|f(\alp)|^4\quad \text{and}\quad \Psi(\alp)=|g(\alp)|^2.$$
When $\alp\in \calM(q,a)$ one has $q+X|q\alp-a|>W$, and thus we obtain the upper bound
\begin{align*}
\int_\grp|f(\alp)^4g(\alp)^2|\d\alp &\ll X\Bigl( (\log X)\int_0^1|g(\alp)|^2\d\alp +Y^\eps W^{\eps-1}|g(0)|^2\Bigr) \\
&\ll XY\log X+XY^{2+\eps}W^{\eps-1}.
\end{align*}
\end{proof}

We are now equipped to complete the proof of Theorem \ref{theorem1.2}. By applying Schwarz's inequality to the integral on the right hand side of (\ref{2.11}), we obtain the estimate
$$Z^*Y(\log W)^{-1}\ll \Bigl( \int_\grp |f(\alp)^4g(\alp)^2|\d\alp \Bigr)^{1/2} \Bigl( \int_0^1|K(\alp)|^2\d\alp \Bigr)^{1/2} .$$
Thus, by Parseval's identity in combination with Lemma \ref{lemma5.2}, we find that
$$Z^*Y(\log W)^{-1}\ll (Z^*)^{1/2}(XY\log X+XY^{2+\eps}W^{\eps-1})^{1/2},$$
whence
$$Z^*\ll XY^{-1}(\log W)^2(\log X)+XY^\eps (\log W)^2W^{\eps-1}.$$
On recalling the definition (\ref{2.2}) of $W$ together with (\ref{2.9}), we discern that
$$Z\ll XY^{-1}(\log X)(\log \log X)^2+XY^\eps(\log X)^{\eps-1/5}.$$
Thus, whenever $\del$ is a positive number with $\del<\tfrac{1}{6}$, and
$$(\log X)(\log \log X)^2\le Y\le (\log X)^{1+\del},$$
one finds that
$$Z\ll XY^{-1}(\log X)(\log \log X)^2+X(\log X)^{(2+\del)\eps-1/5}.$$
It therefore follows that $E(X;Y)\ll XY^{-1}(\log X)(\log \log X)^2$, thereby confirming the conclusion of Theorem \ref{theorem1.2}.\vskip.1cm

We observe that by refining the conclusion of Lemma \ref{lemma5.1} by making use of sharper estimates for the divisor function, one may replace the factor $H^\eps$ therein by one of the shape $\exp(c\log H/\log \log H)$, for a suitable $c>0$. Taking
$$W=\exp(4c(\log \log X)/(\log \log \log X))$$
in the argument above, one then finds that whenever
$$(\log X)(\log \log X)^2(\log \log \log X)^{-2}\le Y\le (\log X)(\log \log X)^3,$$
then
\begin{align*}
Z\ll &\,XY^{-1}(\log X)(\log \log X)^2(\log \log \log X)^{-2}\\
&\,+X\exp(-2c(\log \log X)/(\log \log \log X))(\log \log X)^2(\log \log \log X)^{-2}.
\end{align*}
Then Corollary \ref{corollary1.3} may be replaced with the following result.

\begin{corollary}\label{corollary5.3} Let $\eps$ be a positive number. Then for almost all natural numbers $n$ with $4\nmid n$
 and $n\not\equiv 7\mmod{8}$, the Diophantine equation (\ref{1.1}) possesses a solution $\bfx\in \dbN^3$ 
in which
$$x_3\le (\log n)(\log \log n)^2(\log \log \log n)^{\eps-2}.$$
\end{corollary}

\section{Two squares and two microsquares}
The arguments contained in the previous sections are easily modified to accommodate the analogous problem 
in which the single microsquare of \S\S2-5 is replaced by two microsquares. When $Y$ is positive, we now denote by $R_0(n;Y)$ the number of representations of the integer $n$ in the shape (\ref{1.2}) with $\bfx\in \dbN^4$ and $\max\{x_3,x_4\}\le Y$. Also, write $E_0(X;Y)$ for the number of integers $n$, with $X/2<n\le X$ and $8\nmid n$, having no such representation, so that $R_0(n;Y)=0$. The conclusion of Theorem \ref{theorem1.5} is an immediate consequence of the estimate contained in the following theorem.

\begin{theorem}\label{theorem6.1}
Suppose that $0<\del<\tfrac{1}{11}$. Then whenever
$$(\log X)^{1/2}(\log \log X)^{3/2}\le Y\le (\log X)^{1/2+\del},$$
one has $E_0(X;Y)\ll XY^{-2}(\log X)(\log \log X)^3$.
\end{theorem}

\begin{proof} We now write $\calZ(X)$ for the set of integers $n$ with $X/2<n\le X$ and $8\nmid n$ for which there are no solutions $\bfx\in \dbN^4$ to the Diophantine equation (\ref{1.2}) with $\max\{x_3,x_4\}\le Y$. Write $Z=\text{card}(\calZ(X))$. Then the argument of \S2 is readily adapted to confirm that there exists a set $\calZ^*(X)$ contained in $\calZ(X)$ with the property that, for all $n\in \calZ^*(X)$ one has
$$\int_\grP f(\alp)^2g(\alp)^2e(-n\alp)\d\alp \gg Y^2(\log W)^{-1},$$
and for which
\begin{equation}\label{6.1}
Z\le \text{card}(\calZ^*(X))+XW^{\eps-1}.
\end{equation}
Writing $Z^*$ for $\text{card}(\calZ^*(X))$, we likewise obtain the bound
$$Z^*Y^2(\log W)^{-1}\ll \int_\grp |f(\alp)^2g(\alp)^2K(\alp)|\d\alp .$$
Applying Schwarz's inequality as in the conclusion of \S5, we deduce that
\begin{equation}\label{6.2}
Z^*Y^2(\log W)^{-1}\ll (Z^*)^{1/2}\Bigl( \int_\grp |f(\alp)^4g(\alp)^4|\d\alp \Bigr)^{1/2}.
\end{equation}
But the argument of the proof of Lemma \ref{lemma5.2}, combined with the familiar estimate
$$\int_0^1|g(\alp)|^4\d\alp \ll Y^2\log Y,$$
delivers the bound
\begin{align*}
\int_\grp |f(\alp)^4g(\alp)^4|\d\alp &\ll X\Bigl( (\log X)\int_0^1|g(\alp)|^4\d\alp +Y^\eps W^{\eps-1}|g(0)|^4\Bigr) \\
&\ll XY^2(\log X)(\log Y)+XY^{4+\eps}W^{\eps-1}.
\end{align*}
Thus we conclude from (\ref{6.2}) that
$$Z^*\ll XY^{-2}(\log X)(\log Y)(\log W)^2+XY^\eps (\log W)^2W^{\eps-1}.$$
Recalling the definition (\ref{2.2}) of $W$, it therefore follows from (\ref{6.1}) that
$$Z\ll XY^{-2}(\log X)(\log \log X)^2(\log Y)+XY^\eps (\log X)^{\eps-1/5}.$$
Let $\del$ be a positive number with $\del<\tfrac{1}{11}$. Then whenever
$$(\log X)^{1/2}(\log \log X)^{3/2}\le Y\le (\log X)^{1/2+\del},$$
one sees that
$$Z\ll XY^{-2}(\log X)(\log \log X)^3+X(\log X)^{(2+\del)\eps-1/5}.$$
The conclusion of Theorem \ref{theorem6.1} follows at once.
\end{proof}

\bibliographystyle{amsbracket}
\providecommand{\bysame}{\leavevmode\hbox to3em{\hrulefill}\thinspace}

\end{document}